\documentclass[oneside,english]{amsart}
\usepackage[T1]{fontenc}
\usepackage[latin9]{inputenc}
\usepackage{amsthm}
\usepackage{amstext}
\usepackage{amssymb}
\usepackage{esint}

\makeatletter
\numberwithin{equation}{section}
\numberwithin{figure}{section}
\theoremstyle{plain}
\newtheorem{thm}{\protect\theoremname}
  \theoremstyle{plain}
  \newtheorem{lem}[thm]{\protect\lemmaname}

\usepackage{babel}

\makeatother

\usepackage{babel}
  \providecommand{\lemmaname}{Lemma}
\providecommand{\theoremname}{Theorem}

\begin{document}

\title{Local estimates for positive solutions of porous medium equations}

\author{Zhongmin Qian}

\address{Mathematical Institute, University of Oxford, Oxford OX2 6GG, England.}
\email{qianz@maths.ox.ac.uk}

\thanks{The first author would like to thank the support of an ERC grant.}

\author{Zichen Zhang}

\address{Mathematical Institute, University of Oxford, Oxford OX2 6GG, England.}
\email{zhangz@maths.ox.ac.uk}



\begin{abstract}
We derive several new gradient estimates of Aronson-B{\'e}nilan type for
positive solutions of porous medium equation and fast diffusion equation
on a complete manifold that satisfies the curvature dimension condition. 
\end{abstract}
\maketitle

\section{Introduction}

The porous medium equation 
\begin{equation}
u_{t}=\Delta u^{m}\label{eqc-1}
\end{equation}
 where $m\in\left(0,\infty\right)$, has been studied firstly as useful
models of describing the physical processes of liquid flows passing
through a concrete medium. If the physical medium is not homogenous,
then one may model the physical processes in terms of porous medium
equations in a manifold. The curvature effect of the manifold will
have an impact on the fluid flow through the non-homogenous medium.
On the other hand, the porous medium equation is an archetypical example
for the study of degenerate parabolic equations, and there is a large
number of papers dealing with different aspects of the porous medium
equations, see for example J. L. V{\'a}zquez's monographs \cite{MR2282669}
\cite{MR2286292}.

For a general parameter $m>0$, only positive solutions to the porous
medium equation are interesting, so in this paper we aim to establish
a priori estimates for positive solutions. Let $u$ be a positive
solution to (\ref{eqc-1}), which may be rewritten as the following
\begin{equation}
u_{t}=mu^{m-1}\Delta u+m\left(m-1\right)u^{m-2}\left|\nabla u\right|^{2}\text{.}\label{epc2}
\end{equation}
Observe, in the non-linear case i.e. $m\neq1$, the diffusion coefficient
in (\ref{epc2}), $mu^{m-1}$, depends on the solution $u$. The degeneracy
of the diffusion part in (\ref{epc2}) happens if $\inf u=0$ when
$m>1$ and $\sup u=\infty$ when $m<1$. If $u$ is bounded away from
$0$ and bounded from above, the differential operator $mu^{m-1}\Delta$
is uniformly elliptic. Only in this case, the theory of fully non-linear
equations applies, and the regularity of the solutions follows from
Krylov-Safonov's theory \cite{MR1814364}.

From this point of view, it is an interesting problem to derive, a
priori space derivative estimates for a positive solution $u$ in
terms of its time derivative, the time parameter $t$, and, if it
is possible, the solution $u$ only. Otherwise we wish to establish
estimates depending on its time derivative, time parameter $t$ and
$\sup u$ for the case $m>1$ (resp. $\inf u$ for the case $m<1$).

In a seminal short note \cite{MR524760}, Aronson and B{\'e}nilan established
a powerful estimate, now under the name of gradient estimate, for
a positive solution $u$ of the porous medium equation on $\mathbb{R}^{n}$.
Explicitly, if $m>\frac{n-2}{n}$, $v=f+\frac{m}{m-1}$, then 
\begin{equation}
-\Delta f\leq\frac{N}{2t}\label{epc3}
\end{equation}
 where 
\begin{equation}
f=m\frac{u^{m-1}-1}{m-1}\text{, \ }\frac{2}{N}=\frac{2}{n}+\left(m-1\right)\text{.}\label{epc4}
\end{equation}
 (\ref{epc4}) is equivalent to the following type gradient estimate:
\begin{equation}
\frac{\left|\nabla f\right|^{2}}{\left(m-1\right)f+m}-\frac{f_{t}}{\left(m-1\right)f+m}\leq\frac{N}{2t}\text{.}\label{epc5}
\end{equation}

The striking feature of (\ref{epc5}) lies in the fact that the estimate
is independent of any bounds of $u$, and therefore is useful in addressing
the regularity of the solutions. This fundamental estimate was later
employed in \cite{MR726106} for the study of existence theory, for
obtaining $L_{loc}^{\infty}\left(\mathbb{R}^{n}\right)$ estimates
in \cite{MR797051}, and for obtaining regularity results for the
free boundary of solutions in \cite{MR891781}.

The story of the porous medium equations in a non-homogenous medium
began with the fundamental paper \cite{MR834612} in 1986 by P. Li
and S. T. Yau, in which they derived the celebrated gradient estimate
for the linear heat equation on manifold, i.e. the case $m=1$.

Suppose that $u$ is a positive solution to the heat equation $u_{t}=\Delta u$
in $\left(0,\infty\right)\times M$, where $M$ is a complete manifold
of dimension $n$, $\Delta$ is the Laplace-Beltrami operator associated
with a complete metric $\left(g_{ij}\right)$ (which accounts for
the non-homogenous continued medium). In a seminal paper \cite{MR834612},
Li and Yau showed that, if the Ricci curvature $R_{ij}$ of $\left(g_{ij}\right)$
is non-negative (for an explanation of the Ricci curvature, see below)
\begin{equation}
\left|\nabla f\right|^{2}-f_{t}\leq\frac{n}{2t}\label{epc6}
\end{equation}
 where 
\[
f=\lim_{m\rightarrow1}m\frac{u^{m-1}-1}{m-1}=\ln u
\]
 is the Hopf transformation of $u$. In fact, they established, under
the condition that $R_{ij}\geq-K$ for some constant $K\geq0$, the
following estimate 
\begin{equation}
\left|\nabla f\right|^{2}-\alpha f_{t}\leq\frac{n\alpha^{2}}{2t}+\frac{n\alpha^{2}K}{2\left(\alpha-1\right)},\ \forall\alpha>1.\label{epc7}
\end{equation}
Here, the significant feature is that the quantity $\left|\nabla f\right|^{2}-\alpha f_{t}$
is controlled by the curvature bound, but independent of any bounds
of the solution $u$ itself, which is one of the special feature of
the linear case. The ellipticity of the operator $\Delta$ is somehow
dominated by the Ricci tensor, but independent of the solution $u$.

There are attempts to extend the Aronson-B{\'e}nilan estimate to solutions
of porous medium equations on manifolds, or equivalently extend Li-Yau's
estimate for non-linear parabolic equations in the presence of curvature.

In \cite{MR2487898}, P. Lu, L. Ni, J. L. V{\'a}zquez and C. Villani studied
the porous medium equation on a complete manifold, and established
a gradient estimate in terms of the superemum of the solution for
the case that $m>1$. Explicitly, let $B\left(\mathcal{O},2R\right)$
denote a geodesic ball with centre $\mathcal{O}$ and radius $2R>0$.
Assume that $u$ is a positive solution to (\ref{eqc-1}) on $B\left(\mathcal{O},2R\right)\times\left[0,T\right]$
and that the Ricci curvature $R_{ij}\geq-K^{2}$ on $B\left(\mathcal{O},2R\right)$
for some $K\geq0$. Let $v=\frac{m}{m-1}u^{m-1}$. If $m>1$ and $\beta>1$,
it holds on $B\left(\mathcal{O},R\right)\times\left[0,T\right]$ that
\[
\frac{\left|\nabla v\right|^{2}}{v}-\beta\frac{v_{t}}{v}\leq\alpha\beta^{2}\left(\frac{1}{t}+C_{2}K^{2}v_{\max}^{2R,T}\right)+\alpha\beta^{2}\frac{v_{\max}^{2R,T}}{R^{2}}C_{1}
\]
 where 
\[
v_{\max}^{2R,T}=\max_{B\left(\mathcal{O},2R\right)\times\left[0,T\right]}v.
\]
 If $m\in\left(1-\frac{2}{n},1\right)$, they proved that on $B\left(\mathcal{O},R\right)\times\left[0,T\right]$,
for any $\gamma\in\left(0,1\right)$, 
\[
\frac{\left|\nabla v\right|^{2}}{v}-\gamma\frac{v_{t}}{v}\geq\frac{\alpha\gamma^{2}}{C_{3}}\left(\frac{1}{t}+C_{4}\sqrt{C_{3}}K^{2}\bar{v}_{\max}^{2R,T}\right)+\frac{\alpha\gamma^{2}}{C_{3}}\frac{\bar{v}_{\max}^{2R,T}}{R^{2}}C_{5}
\]
 with 
\[
\bar{v}_{\max}^{2R,T}=\max_{B\left(\mathcal{O},2R\right)\times\left[0,T\right]}\left(-v\right)\text{.}
\]

In \cite{huang2011gradient}, G. Huang, Z. Huang and H. Li obtained
similar estimates. Note that these gradient bounds depend on the bounds
of the solution $u$, and therefore are not local.

S.T. Yau \cite{MR1305712} established a similar type of gradient
bounds depending on derivatives of initial data for degenerate parabolic
equations of the form 
\[
u_{t}=\Delta\left(F\left(u\right)\right)
\]
 with $F\in C^{2}\left(0,\infty\right)$ and $F^{\prime}>0$. In particular,
as explained in \cite{MR2392508}, Yau's result implies that for any
function $c\left(t\right)\in C^{1}\left(0,\infty\right)$ satisfying
\[
\begin{cases}
c\left(t\right)\leq0\\
c^{\prime}\left(t\right)\geq0\\
\left|\nabla v\right|^{2}-2v_{t}+2m\left(\frac{m-1}{m}v\right)^{\frac{m-2}{m-1}}\leq c\left(t\right) & \mbox{at }t=0
\end{cases}
\]
 it holds for all $t>0$ that 
\[
\left|\nabla v\right|^{2}-2v_{t}+2m\left(\frac{m-1}{m}v\right)^{\frac{m-2}{m-1}}\leq c\left(t\right).
\]
 The gradient estimate of Yau therefore depends on the initial data
$u$.

In this paper, we consider a positive solution $u$ of the equation
\begin{equation}
u_{t}=Lu^{m}\text{ in }\left[0,\infty\right)\times M\label{eq:poro-e1}
\end{equation}
 where $m>0$ is a real exponent, and $L=\Delta+V$ with some smooth
vector field $V$. The Bakry-mery curvature-dimension condition $CD\left(n,-K\right)$
is assumed to hold for $L$, where $K\geq0$ and $m\neq1$. (See below
for an explanation of $CD\left(n,-K\right)$.)

One of the contributions in the present paper is the gradient estimate
for porous medium equation on a complete manifold, which is independent
of the bounds of $u$, for the case $m<1$. Our estimate is a generalization
of Aronson-B{\'e}nilan\textquoteright{}s estimate and Li-Yau's estimate.
\begin{thm}
\label{th-a} Let $M$ be a complete manifold of dimension $n$ and
the curvature-dimension condition $CD\left(n,-K\right)$ hold for
some constant $K\geq0$. Let $u$ be a positive solution to the porous
medium equation \eqref{eq:poro-e1} in $\left(0,\infty\right)\times M$.
Let $N>0$ be defined by $\frac{2}{N}=\frac{2}{n}+m-1$.

1) If $m>\max\left\{ \frac{1}{2},1-\frac{2}{n}\right\} $ and $m<1$,
then 
\begin{equation}
m\frac{\left|\nabla u\right|^{2}}{u^{3-m}}-\frac{u_{t}}{u}\leq\frac{N}{2t}+\frac{2Km}{\left(1-m\right)\left(2m-1\right)}u^{m-1}\text{.}\label{eq-d1}
\end{equation}

2) If $K=0$, $m>1-\frac{2}{n}$, and if $c\in\left[0,\infty\right]$
such that $-\frac{m}{m-1}Lu^{m-1}\left(0,\cdot\right)\leq c$, then
\begin{equation}
m\frac{\left|\nabla u\right|^{2}}{u^{3-m}}-\frac{u_{t}}{u}\leq\frac{c}{\frac{N}{2t}c+1}\text{.}\label{eqd-2}
\end{equation}

\end{thm}
If $K\neq0$, and $m>1$, we are unable to derive an estimate with
the following form 
\[
\left|\nabla u\right|^{2}\leq Q\left(t,u_{t},u\right)
\]
and we do not believe it is possible for the reason which will become
clear, which makes striking difference from the linear case. In this
aspect, we therefore seek for best possible estimates for $\left|\nabla u\right|{}^{2}$
in terms of $t,u_{t},u$ and $\sup u$ as well. To state our results,
let us introduce several notations.

Let $N>0$ and $R\geq0$. Define 
\begin{equation}
w\left(t,y\right)=\frac{4t}{N}\sqrt{NR}\sqrt{y+\frac{NR}{4}}\label{eq-f1}
\end{equation}
 and $Q\left(t,y\right)=C\left(t,y\right)+y$, where $C$ is the solution
to the differential equation 
\[
\frac{d}{dt}C+\frac{2}{N}C^{2}-2R\left(C+y\right)=0\text{, \ }C\left(0\right)=\infty\text{.}
\]
 That is 
\begin{equation}
Q\left(t,y\right)=\frac{NR}{2}+y+\sqrt{NR}\sqrt{y+\frac{NR}{4}}\coth\frac{w\left(t,y\right)}{2}\text{.}\label{eq-f3}
\end{equation}

We are now in a position to state our result for the case that $K>0$
and $m>1$.
\begin{thm}
\label{th-b} Let $M$ be a complete manifold $M$ of dimension $n$
and the curvature-dimension condition $CD\left(n,-K\right)$ hold
for some constant $K\geq0$. Let $u$ be a positive solution to the
porous medium equation\eqref{eq:poro-e1} in $\left(0,\infty\right)\times M$.
Let $f=m\left(u^{m-1}-1\right)/\left(m-1\right)$ and $U=\left(m-1\right)f+m$.
Let $R=\sup\left\{ KU\right\} $. Then 
\[
\frac{\left|\nabla f\right|{}^{2}}{U}\leq Q\left(t,\frac{f_{t}}{U}\right)\text{ on }\frac{f_{t}}{U}>-\frac{NR}{4}
\]
 and 
\[
\frac{\left|\nabla f\right|{}^{2}}{U}-\frac{f_{t}}{U}\leq\frac{N}{2t}+\frac{NR}{2}\text{ on }\frac{f_{t}}{U}\leq-\frac{nR}{4}\text{.}
\]
 
\end{thm}
Note that if $m=1$ then $U=1$, so $R$ reduces in the linear case
to $K$. If $K=0$ then $R=0$, thus our estimate reduces to Aronson-B{\'e}nilan\textquoteright{}s
estimate.

\section{Notations and geometric assumptions}

Let us begin with a description of our precise work setting. Let $M$
be a complete manifold of dimension $n$, and the complete metric
$ds^{2}=\sum_{i,j=1}^{d}g_{ij}dx^{i}dx^{j}$ in a local coordinate
system. Thus $\left(g_{ij}\right)$ is a symmetric, positive definite
type $\left(0,2\right)$ smooth tensor field on $M$. Let $\Delta$
denote the Laplace-Beltrami operator on $M$ given in a local coordinate
system by 
\begin{equation}
\Delta=\frac{1}{\sqrt{g}}\sum_{i,j}\frac{\partial}{\partial x^{i}}g^{ij}\sqrt{g}\frac{\partial}{\partial x^{j}}\label{04-a1}
\end{equation}
 which is an elliptic differential operator of second order, where
$\left(g^{ij}\right)$ is the inverse matrix of $\left(g_{ij}\right)$,
and $g=\det\left(g_{ij}\right)$. Let $V$ be a smooth vector field,
and let $L=\Delta+V$.

The geometric condition we are going to use, which is actually the
only place where the geometry of the complete manifold $\left(M,g_{ij}\right)$
enters into our study, is the curvature-dimension condition, a concept
that is closely related to Ricci curvature lower bound. 

The metric $\left(g_{ij}\right)$ may be recovered from the elliptic
operator by means of 
\begin{eqnarray*}
\Gamma\left(u,v\right) & = & \frac{1}{2}\left(L\left(uv\right)-uLv-vLu\right)\\
 & = & \sum_{i,j=1}^{d}g^{ij}\frac{\partial u}{\partial x^{i}}\frac{\partial v}{\partial x^{j}}=\left\langle \nabla u,\nabla v\right\rangle \text{,}
\end{eqnarray*}
 then the curvature operator $\Gamma_{2}\left(u,v\right)$ is defined
by iterating $\Gamma$ and given by 
\begin{equation}
\Gamma_{2}\left(u,v\right)=\frac{1}{2}\left(L\Gamma\left(u,v\right)-\Gamma\left(u,Lv\right)-\Gamma\left(v,Lu\right)\right)\text{.}\label{04-b1}
\end{equation}
 The Bochner identity then gives the explicit formula in terms of
the Ricci curvature tensor $\left(R_{ij}\right)$ associated with
the Riemannian metric $\left(g_{ij}\right)$ and the symmetric part
of the total co-variant derivative $\left(V_{ij}\right)$ of the vector
field $\nabla V$, in a local coordinate system, 
\begin{equation}
\Gamma_{2}\left(u,u\right)=\left|\nabla^{2}u\right|^{2}+\sum_{i,j=1}^{d}\left(R_{ij}-V_{ij}\right)\nabla^{i}f\nabla^{j}f\label{04-b2}
\end{equation}
 where $V_{ij}=\frac{1}{2}\left(\nabla_{i}V_{j}+\nabla_{j}V_{i}\right)$,
and $V=\sum_{i=1}^{d}V^{i}\frac{\partial}{\partial x^{i}}$ is the
vector field represented in terms of coordinate partial differentiation.
It is elementary that 
\begin{equation}
\left|\nabla^{2}u\right|^{2}\geq\frac{1}{n}\left(\Delta u\right)^{2}\text{.}\label{04-b3}
\end{equation}
 Therefore, in the case that $V=0$, we have the following curvature-dimension
inequality 
\begin{equation}
\Gamma_{2}\left(u,u\right)\geq\frac{1}{n}\left(\Delta u\right)^{2}-K\left|\nabla u\right|^{2}\text{.}\label{04-b4}
\end{equation}
 Hence, for a general vector field $V$, the geometric condition on
$L$ should be the following curvature-dimension inequality $CD\left(n,-K\right)$
\begin{equation}
\Gamma_{2}\left(u,u\right)\geq\frac{1}{n}\left(Lu\right)^{2}-K\left|\nabla u\right|^{2}\label{04-b5}
\end{equation}
 for some $n>0$ and a constant $K$.

Under this curvature-dimension condition, we will study the porous
medium equation 
\begin{equation}
u_{t}=Lu^{m}\text{ in }\left[0,\infty\right)\times M,\label{eq:poro}
\end{equation}
where $m>0$ is a real exponent.

Throughout the paper, we will use $t$ to denote the time parameter
and $\partial_{t}$ to denote the partial derivative with respect
to $t$. For simplicity, if no confusion may arise, then $h_{t}$
denotes $\partial_{t}h$.

\section{fundamental differential inequalities}

Given a positive solution $u$ of equation \eqref{eq:poro}, consider
the modified Hopf transformation 
\begin{equation}
f=m\frac{u^{m-1}-1}{m-1}\label{04-a2}
\end{equation}
 which recovers to the logarithm transform in the sense that 
\[
\lim_{m\rightarrow1}m\frac{u^{m-1}-1}{m-1}=\log u\text{.}
\]
 It is clear that $\left(m-1\right)f+m=mu^{m-1}$, which is positive.
This quantity will play an important role below.

The porous medium equation is then transformed into a parabolic equation
for $f$ : 
\begin{equation}
\left[\left(\left(m-1\right)f+m\right)L-\partial_{t}\right]f=-\left|\nabla f\right|^{2}\text{.}\label{eq:f-heat-g1}
\end{equation}
 The parabolic equation (\ref{eq:f-heat-g1}) for $f$ has the same
form as the linear case. Although the elliptic operator $A$ involved
is much more complicated than the linear case, it however suggests
the possibility for obtaining gradient estimates that are similar
to Aronson-B{\'e}nilan \cite{MR524760} and Li-Yau \cite{MR834612} type.
In fact our arguments from now on only rely on the fact that $f$
is a solution to (\ref{eq:f-heat-g1}) and the fact (or assumption)
that $\left(m-1\right)f+m$ is positive.

The parabolic equation (\ref{eq:f-heat-g1}) may be written as 
\begin{equation}
Lf=Y-X\label{04-a3}
\end{equation}
 where 
\begin{equation}
X=\frac{\left|\nabla f\right|{}^{2}}{\left(m-1\right)f+m}\text{ \ and \ }Y=\frac{f_{t}}{\left(m-1\right)f+m}.\label{04-a4}
\end{equation}

In general, a gradient estimate will take the form $-Lf\leq c\left(t,f,f_{t}\right)$.
Due to a reason which will become clear only later on we wish to point
out that we will seek for a gradient estimate in terms of $X$ and
$Y$ with the following form 
\begin{equation}
X-Y\leq B\left(t,U,Y\right)\label{eq:gcform1}
\end{equation}
 for some function $B$ which can be either computed explicitly or
can be determined by simple ordinary differential equations (see below
for details), where $U=\left(m-1\right)f+m$. As long as the problem
of determining $B\left(t,U,Y\right)$ has been settled, the idea then
is to apply the maximum principle to test function 
\begin{equation}
G=t\left[X-Y-B\left(t,U,Y\right)\right]\label{eq:gfunc1}
\end{equation}
to conclude that $G\leq0$, which in turn yields the desired estimate
(\ref{eq:gcform1}). To this end we will consider the heat operator
$\left(A-\partial_{t}\right)$ applying to $G$: 
\begin{equation}
\left(A-\partial_{t}\right)G=-\frac{1}{t}G+t\left(A-\partial_{t}\right)\left(X-Y\right)-t\left(A-\partial_{t}\right)B\left(t,U,Y\right)\label{eq:heatforg}
\end{equation}
where $A$ is an elliptic differential operator, yet need to be determined
as well. So our first task is to derive parabolic equations that $X$
and $Y$ must satisfy. 

Let us start with the parabolic equation (\ref{eq:f-heat-g1}) for
$f$, which may be written as 
\begin{equation}
f_{t}=\left(\left(m-1\right)f+m\right)\left(Lf+X\right)\text{ .}\label{ft-a-1}
\end{equation}

To derive the evolution equations for $Y$ and $X$, first differentiate
the equation (\ref{eq:f-heat-g1}) in $t$ to obtain 
\[
\partial_{t}f_{t}=\left(\left(m-1\right)f+m\right)Lf_{t}+2\left(\left(m-1\right)f+m\right)\left\langle \nabla f,\nabla f_{t}\right\rangle \text{.}
\]
 By using product rule we can work out the time derivative of $Y$
as 
\[
\partial_{t}Y=\left(\left(m-1\right)f+m\right)^{-1}\partial_{t}f_{t}-\left(m-1\right)\left(\left(m-1\right)f+m\right)^{-2}\left(f_{t}\right)^{2}.
\]
 Together with the equation for $\partial_{t}f_{t}$ we thus deduce
that 
\begin{eqnarray}
\partial_{t}Y & = & Lf_{t}+2\left(\left(m-1\right)f+m\right)^{-1}\left\langle \nabla f,\nabla f_{t}\right\rangle \nonumber \\
 &  & +\left(m-1\right)\left(\left(m-1\right)f+m\right)^{-1}f_{t}Lf\nonumber \\
 &  & -\left(m-1\right)\left(\left(m-1\right)f+m\right)^{-2}\left(f_{t}\right)^{2}\text{.}\label{eq:y-14-01}
\end{eqnarray}
 Substituting $f_{t}$ in terms of $Y$ and applying chain rule accordingly
yield 
\begin{eqnarray}
Lf_{t} & = & \left(\left(m-1\right)f+m\right)LY+2\left(m-1\right)\left\langle \nabla f,\nabla Y\right\rangle +\left(m-1\right)YLf\nonumber \\
 & = & \left(\left(m-1\right)f+m\right)LY+2\left(m-1\right)\left\langle \nabla f,\nabla Y\right\rangle \nonumber \\
 &  & +\left(m-1\right)Y\left(Y-X\right).\label{eq:ft-14-01}
\end{eqnarray}
 Similarly, we have 
\begin{equation}
\nabla f_{t}=\left(\left(m-1\right)f+m\right)\nabla Y+\left(m-1\right)Y\nabla f\text{ .}\label{eq:ft-14-02}
\end{equation}
 By plugging these equations into (\ref{eq:y-14-01}) we obtain 
\[
\partial_{t}Y=\left(\left(m-1\right)f+m\right)LY+2m\left\langle \nabla f,\nabla Y\right\rangle +\left(m-1\right)Y^{2}\text{.}
\]
 Define an elliptic differential operator of second order as 
\begin{equation}
A=\left(\left(m-1\right)f+m\right)L+2m\nabla f.\nabla,\label{eq:L-eq-a2}
\end{equation}
then the evolution equation for $Y$ can be neatly expressed as 
\begin{equation}
\left(A-\partial_{t}\right)Y=-\left(m-1\right)Y^{2}.\label{Y-A-2}
\end{equation}
 Therefore, (\ref{eq:L-eq-a2}) is the right choice of $A$ we have
been looking for. 

Next, let us consider $(A-\partial_{t})X$. Again we begin with the
time derivative of $X$: 
\[
\partial_{t}X=2\left(\left(m-1\right)f+m\right)^{-1}\nabla f.\nabla f_{t}-\left(m-1\right)YX\text{.}
\]
 By using (\ref{eq:f-heat-g1}) we obtain 
\begin{equation}
\partial_{t}X=2\langle\nabla f,\nabla X\rangle+2\left\langle \nabla f,\nabla Lf\right\rangle +m\left(m-1\right)XY\label{eq:tX-14-01}
\end{equation}
 and 
\begin{eqnarray*}
LX & = & L\left[\left(\left(m-1\right)f+m\right)^{-1}|\nabla f|^{2}\right]\\
 & = & -2\left(m-1\right)\left(\left(m-1\right)f+m\right)^{-1}\langle\nabla f,\nabla X\rangle\\
 &  & +\left(\left(m-1\right)f+m\right)^{-1}L\left|\nabla f\right|^{2}\\
 &  & -\left(m-1\right)\left(\left(m-1\right)f+m\right)^{-1}X\left(Y-X\right)
\end{eqnarray*}
 so that 
\[
AX=2\langle\nabla f,\nabla X\rangle+L\left|\nabla f\right|^{2}-\left(m-1\right)X\left(Y-X\right)\text{.}
\]
 Therefore 
\begin{equation}
\left(A-\partial_{t}\right)X=2\Gamma_{2}\left(f,f\right)+\left(m-1\right)\left(X-Y\right)^{2}-\left(m-1\right)Y^{2}\text{.}\label{X-A-2}
\end{equation}

It is also convenient to consider $U=\left(m-1\right)f+m$ as a basic
function when we seek for gradient estimates. In the case for linear
equations, $U$ reduces to a constant which explains why the solution
does not appear in Li-Yau type gradient estimates. It follows from
(\ref{eq:f-heat-g1}) the evolution equation for $U$ which is given
by 
\begin{equation}
\left(A-\partial_{t}\right)U=\left(2m-1\right)m(m-1UX\text{.}\label{eq:f-eq}
\end{equation}
The two functions related to the space and time derivatives of $f$,
$X$ and $Y$, may be rewritten as 
\begin{equation}
X=\frac{\left|\nabla f\right|^{2}}{U}\text{ and }Y=\frac{f_{t}}{U}.\label{eq:def-X-Y}
\end{equation}
 According to their evolution equations (\ref{Y-A-2}) and (\ref{X-A-2}),
it is a good idea to introduce a new variable $Z=X-Y=-Lf$, whose
evolution equation can be obtained by taking the difference of (\ref{Y-A-2})
and (\ref{X-A-2}) as follows 
\begin{equation}
\left(A-\partial_{t}\right)Z=2\Gamma_{2}\left(f,f\right)+\left(m-1\right)Z^{2}\text{.}\label{Z-eq1}
\end{equation}

Suppose that the following curvature-dimension condition $CD(n,-K)$
holds 
\begin{eqnarray*}
\Gamma_{2}\left(f,f\right) & \geq & \frac{1}{n}\left(Lf\right){}^{2}-K\left|\nabla f\right|^{2}\\
 & = & \frac{1}{n}Z^{2}-KU\left(Z+Y\right)^{2}\text{.}
\end{eqnarray*}
 It follows therefore from (\ref{Z-eq1}) that 
\begin{equation}
\left(A-\partial_{t}\right)Z\geq\frac{2}{N}Z^{2}-2KU\left(Z+Y\right)^{2}\label{Z-eq2}
\end{equation}
 where 
\[
\frac{2}{N}=\frac{2}{n}+m-1\text{.}
\]
 (\ref{Z-eq2}) is a fundamental inequality, which, together with
the evolution equations for $U$ and $Y$, will lead to different
kinds of gradient estimates in the sequel.

\section{Local gradient estimates}

It seems a long and uninitiated procedure to derive various differential
inequalities, and up to now we have not derived any useful estimates.
From now on however we are going to deliver useful estimates for $u$
(in terms of its Hopf transformation $f$), based solely on the following
differential inequalities 
\begin{equation}
\left\{ \begin{array}{ccc}
(A-\partial_{t})U & = & (2m-1)(m-1)U(Y+Z)\text{,}\\
(A-\partial_{t})Y & = & -(m-1)Y^{2}\text{, \ \ \ \ \ \ \ \ \ \ \ }\\
(A-\partial_{t})Z & \geq & \frac{2}{N}Z^{2}-2KU(Y+Z)\text{ \ \ \ \ \ }
\end{array}\right.\label{3A-e1}
\end{equation}
 where $Z=X-Y$, $U=(m-1)f+m$. We remark the facts that $U\geq0$
and $Z+Y\geq0$.
\begin{proof}[Proof of \eqref{eqd-2} in Theorem \ref{th-a}]

When $K=0$, the third differential inequality in (\ref{3A-e1}) reduces
to 
\begin{equation}
(A-\partial_{t})Z\geq\frac{2}{N}Z^{2}.\label{eq:20140227-1}
\end{equation}
Assume at time $0$
\[
Z=X-Y\leq c.
\]
We apply maximum principle to 
\[
J=\left(\frac{2t}{N}+\frac{1}{c}\right)Z.
\]
On the one hand, by \eqref{eq:20140227-1} we have 
\begin{eqnarray*}
\left(A-\partial_{t}\right)J & = & \left(\frac{2t}{N}+\frac{1}{c}\right)\left(A-\partial_{t}\right)Z-\frac{2}{N}Z\\
 & \geq & \left(\frac{2t}{N}+\frac{1}{c}\right)\frac{2}{N}Z^{2}-\frac{2}{N}Z\\
 & = & \frac{2}{N}Z\left(J-1\right),
\end{eqnarray*}
which ensures that the maximum of $J$ over $\left(0,\infty\right)\times M$,
if exists, can not be greater than $1$. On the other hand, at time
$0$, since $Z\leq C$, we have $J\leq1$. Therefore, we conclude
that 
\[
X-Y\leq\left(\frac{2t}{N}+\frac{1}{c}\right)^{-1},
\]
which is \eqref{eqd-2} in Theorem \ref{th-a}.
\end{proof}
In order to derive (\ref{eq-d1}) in Theorem \ref{th-a} and Theorem
\ref{th-b}, we need several facts, which we are going to establish
now. 

When $K>0$, the differential inequality for $Z$ involves $Y$ and
$U$. Hence we look for gradient estimates in the form 
\[
Z\leq B\left(t,U,Y\right),
\]
 and we wish to prove this estimate by applying the maximum principle
to the test function $G=t\left(Z-B(t,U,Y)\right)$. The idea is to
apply the heat operator $(A-\partial_{t})$ to $G$. Here we recall
that 
\[
A=UL+2m\nabla f.\nabla
\]
 is an elliptic operator of second order. 
\begin{lem}
\label{lem-1} Assume that $B\left(t,U,Y\right)$ is smooth in $\left(t,U,Y\right)$
and is convex in $\left(U,Y\right)$. Let $G=t\left(Z-B\left(t,U,Y\right)\right)$.
Then it holds that 
\begin{eqnarray}
\left(A-\partial_{t}\right)G & \geq & \frac{2}{Nt}G^{2}+\left\{ \frac{4}{N}B-\frac{1}{t}-\left(2K+(m-1)\left(2m-1\right)B_{U}\right)U\right\} G\nonumber \\
 &  & +\left\{ B_{t}+\frac{2}{N}B^{2}-\left(2K+(m-1)\left(2m-1\right)B_{U}\right)U(Y+B)\right\} t\nonumber \\
 &  & +(m-1)tY^{2}B_{Y}\text{.}\label{eq:20140228-1}
\end{eqnarray}
\end{lem}
\begin{proof}
Define $F=Z-B\left(t,U,Y\right)$. By product rule 
\begin{eqnarray*}
\left(A-\partial_{t}\right)F & = & \left(A-\partial_{t}\right)Z-\left(A-\partial_{t}\right)B(t,U,Y)\\
 & = & \left(A-\partial_{t}\right)Z+B_{t}-B_{Y}\left(A-\partial_{t}\right)Y-B_{U}\left(A-\partial_{t}\right)U\\
 &  & -U\left\{ B_{Y^{2}}|\nabla Y|^{2}+2B_{YU}\langle\nabla U,\nabla Y\rangle+B_{U^{2}}|\nabla U|^{2}\right\} \text{.}
\end{eqnarray*}
 For smooth function $B(t,U,Y)$ which is convex in $(U,Y)$, 
\[
B_{Y^{2}}|\nabla Y|^{2}+2B_{YU}\langle\nabla U,\nabla Y\rangle+B_{U^{2}}|\nabla U|^{2}\leq0\text{.}
\]
 Hence we have 
\begin{equation}
\left(A-\partial_{t}\right)F\geq\left(A-\partial_{t}\right)Z+B_{t}-B_{Y}\left(A-\partial_{t}\right)Y-B_{U}\left(A-\partial_{t}\right)U\label{04-02-c}
\end{equation}
 which yields that 
\begin{eqnarray}
\left(A-\partial_{t}\right)G & = & -\frac{1}{t}G+t\left(A-\partial_{t}\right)F\nonumber \\
 & \geq & -\frac{1}{t}G+t\left(A-\partial_{t}\right)Z+tB_{t}\nonumber \\
 &  & -tB_{Y}\left(A-\partial_{t}\right)Y-tB_{U}\left(A-\partial_{t}\right)U.\label{04-02-d}
\end{eqnarray}
 By applying the fundamental differential inequalities (\ref{3A-e1})
we obtain 
\begin{eqnarray}
\left(A-\partial_{t}\right)G & \geq & -\frac{1}{t}G+t\left\{ \frac{2}{N}Z^{2}-2KU(Y+Z)\right\} +tB_{t}\nonumber \\
 &  & +(m-1)t\left\{ Y^{2}B_{Y}-(2m-1)U(Y+Z)B_{U}\right\} \text{.}\label{04-02-f}
\end{eqnarray}
 Now replace $Z$ by $\frac{1}{t}G+B=Z$ to obtain 
\begin{eqnarray*}
\left(A-\partial_{t}\right)G & \geq & -\frac{1}{t}G+t\left\{ \frac{2}{N}\left[\frac{1}{t}G+B\right]^{2}-2KU(Y+\frac{1}{t}G+B)\right\} \\
 &  & +tB_{t}+(m-1)t\left\{ Y^{2}B_{Y}-(2m-1)UB_{U}(Y+\frac{1}{t}G+B)\right\} \text{.}
\end{eqnarray*}
 By rearranging the terms we deduce that 
\begin{eqnarray*}
\left(A-\partial_{t}\right)G & \geq & \frac{2}{Nt}G^{2}+\left\{ \frac{4}{N}B-\frac{1}{t}-\left(2K+(m-1)\left(2m-1\right)B_{U}\right)U\right\} G\\
 &  & +\left\{ B_{t}+\frac{2}{N}B^{2}-\left(2K+(m-1)\left(2m-1\right)B_{U}\right)U(Y+B)\right\} t\\
 &  & +(m-1)tY^{2}B_{Y}\text{.}
\end{eqnarray*}

\end{proof}
With this result at hand, we are able to test whether a function $B\left(t,U,Y\right)$
convex in $\left(U,Y\right)$ is an upper bound on $Z$.

Let us consider the linear form 
\[
B(t,U,Y)=\theta+aU+bY
\]
 where $\theta$, $a$ and $b$ are functions in $t$. Then by the
lemma above, 
\begin{eqnarray}
\left(A-\partial_{t}\right)G & \geq & \frac{2}{Nt}G^{2}+\left(\frac{4}{N}\theta-\frac{1}{t}\right)G\nonumber \\
 &  & +2\left(\frac{2b}{N}Y-\lambda_{1}U\right)G+E_{1}t\label{04-3a}
\end{eqnarray}
 where
\begin{eqnarray*}
E_{1} & = & B_{t}+\frac{2}{N}B^{2}-\lambda_{2}((1+b)UY+aU^{2}+\theta U)+(m-1)bY^{2}\\
 & = & B_{t}+\frac{2}{N}\theta^{2}+\left(\frac{4a}{N}-\lambda_{2}\right)\theta U+\left(\frac{2}{N}a-\lambda_{2}\right)aU^{2}\\
 &  & +\frac{4b}{N}\theta Y+\left(\frac{2}{N}b^{2}+(m-1)b\right)Y^{2}+\left(\frac{4}{N}ab-\lambda_{2}(1+b)\right)UY,
\end{eqnarray*}
and 
\begin{eqnarray*}
\lambda_{1} & = & \left(K+\left[\frac{(m-1)\left(2m-3\right)}{2}-\frac{2}{n}\right]a\right)\text{,}\\
\lambda_{2} & = & \left(2K+(m-1)\left(2m-1\right)a\right).
\end{eqnarray*}
Introduce a function $\psi$ and rearrange the differential inequality
(\ref{04-3a}) as 
\begin{eqnarray}
\left(A-\partial_{t}\right)G & \geq & \frac{2}{Nt}\left[G+\frac{Nt}{2}\left(\frac{2b}{N}Y-\lambda_{1}U+\psi\right)\right]^{2}\nonumber \\
 &  & +\left(\frac{4}{N}\theta-\frac{1}{t}-2\psi\right)G+E_{2}t\label{04-3b}
\end{eqnarray}
 where 
\begin{eqnarray*}
E_{2} & = & E_{1}-\frac{N}{2}\left(\frac{2b}{N}Y-\lambda_{1}U+\psi\right)^{2}\\
 & = & E_{1}-\frac{2b^{2}}{N}Y^{2}-\frac{N}{2}\left(\lambda_{1}U-\psi\right)^{2}+2b\left(\lambda_{1}U-\psi\right)Y\\
 & = & B_{t}+\frac{2}{N}\theta^{2}+\left(\frac{4a}{N}-\lambda_{2}\right)\theta U+\left(\frac{2}{N}a^{2}-\lambda_{2}a-\frac{N}{2}\lambda_{1}^{2}\right)U^{2}\\
 &  & +\frac{4b}{N}\theta Y+(m-1)bY^{2}+\left(\frac{4}{N}ab-\lambda_{2}(1+b)+2\lambda_{1}b\right)UY\\
 &  & +N\lambda_{1}U\psi-2b\psi Y-\frac{N}{2}\psi^{2}\text{.}
\end{eqnarray*}
 Insert the values of $\lambda_{1}$ and $\lambda_{2}$ to obtain
\begin{eqnarray}
E_{2} & = & B_{t}+\frac{2}{N}\theta^{2}+\frac{4b}{N}\theta Y+(m-1)bY^{2}\nonumber \\
 &  & +\left(\frac{4}{n}a-(m-1)\left(2m-3\right)a-2K\right)\theta U\nonumber \\
 &  & +\left(\frac{2}{N}a^{2}-(m-1)\left(2m-1\right)a^{2}-2Ka-\frac{N}{2}\lambda_{1}^{2}\right)U^{2}\nonumber \\
 &  & -\left((m-1)\left(2m-1\right)a+2K\right)UY\nonumber \\
 &  & +N\lambda_{1}U\psi-2b\psi Y-\frac{N}{2}\psi^{2}\text{.}\label{eq:04-3c}
\end{eqnarray}

Let us consider the coefficient in front of $U^{2}$ in (\ref{eq:04-3c}),
which is 
\begin{eqnarray*}
\lambda_{3} & = & \frac{2}{N}a^{2}-(m-1)\left(2m-1\right)a^{2}-2Ka-\frac{N}{2}\lambda_{1}^{2}\\
 & = & -\left[(m-1)\left(2m-1\right)a+2K\right]^{2}\frac{N}{8}.
\end{eqnarray*}
 Therefore 
\begin{eqnarray}
E_{2} & = & B_{t}+\frac{2}{N}\theta^{2}+\frac{4b}{N}\theta Y+(m-1)bY^{2}\nonumber \\
 &  & +N\left(K+\left[\frac{(m-1)\left(2m-1\right)}{2}-\frac{2}{N}\right]a\right)\psi U-2b\psi Y-\frac{N}{2}\psi^{2}\nonumber \\
 &  & +\left(\frac{4}{n}a-(m-1)\left(2m-3\right)a-2K\right)\theta U\nonumber \\
 &  & -\left[(m-1)\left(2m-1\right)a+2K\right]^{2}\frac{N}{8}U^{2}\nonumber \\
 &  & -\left((m-1)\left(2m-1\right)a+2K\right)UY\text{. }\label{04-c5}
\end{eqnarray}
Inserting 
\[
B_{t}=\theta^{\prime}+a^{\prime}U+b^{\prime}Y
\]
into it yields 
\begin{eqnarray}
E_{2} & = & \theta^{\prime}+\frac{2}{N}\theta^{2}+\left(\frac{4b}{N}\theta+b^{\prime}\right)Y+(m-1)bY^{2}\nonumber \\
 &  & +N\left(K+\left[\frac{(m-1)\left(2m-1\right)}{2}-\frac{2}{N}\right]a\right)\psi U-2b\psi Y-\frac{N}{2}\psi^{2}\nonumber \\
 &  & +\left(\left[\frac{4}{n}a-(m-1)\left(2m-3\right)a-2K\right]\theta+a^{\prime}\right)U\nonumber \\
 &  & -\frac{N}{8}\left[(m-1)\left(2m-1\right)a+2K\right]^{2}U^{2}\nonumber \\
 &  & -\left((m-1)\left(2m-1\right)a+2K\right)UY\text{. }\label{04-6a}
\end{eqnarray}

Assume that $N=2\left(\frac{2}{n}+m-1\right)^{-1}>0$. According to
(\ref{04-3b}), in order to be able to conclude that $G\leq0$ by
using the maximum principle, we need to choose $\psi$, $a$, $b$
and $\theta$ such that 
\begin{equation}
\frac{4}{N}\theta-2\psi-\frac{1}{t}\geq0,\mbox{ and }E_{2}\geq0.\label{eq:20140228-2}
\end{equation}
We are now in a position to prove (\ref{eq-d1}) in Theorem \ref{th-a}.
\begin{proof}[Proof of (\ref{eq-d1}) in Theorem \ref{th-a}]

By the expression (\ref{04-6a}) for $E_{2}$, the best we can do
is to choose $a$ such that 
\[
(m-1)\left(2m-1\right)a+2K=0
\]
 otherwise we can not control $U^{2}$ for whatever the choice of
$b$, as long as $m\neq1$ (if $m=1$ then $U=1$, the story would
be different then). This is possible only if $(m-1)\left(2m-1\right)\neq0$,
so we are working with this constraint. 

With this particular choice of $a$, we have
\begin{eqnarray}
E_{2} & = & \theta^{\prime}+\frac{2}{N}\theta^{2}+\left(b^{\prime}+4b\frac{\theta}{N}\right)Y+(m-1)bY^{2}\nonumber \\
 &  & -\frac{2}{(m-1)\left(2m-1\right)}\left(\frac{4}{N}\theta-2\psi\right)KU-2b\psi Y-\frac{N}{2}\psi^{2}.\label{eq:E-2}
\end{eqnarray}
Recall that we do not wish to use any upper bound on $U$, and the
only assumption we imposed on $U$ is its positivity. Therefore, to
control $E_{2}$, we need 
\begin{equation}
-\frac{2}{(m-1)\left(2m-1\right)}\left(\frac{4}{N}\theta-2\psi\right)\geq0.\label{eq:20140228-3}
\end{equation}
However, the first inequality in condition (\ref{eq:20140228-2})
implies that $\frac{4}{N}\theta-2\psi$ needs to be positive. Hence
to ensure (\ref{eq:20140228-3}) to hold, we need $-\frac{2K}{(m-1)\left(2m-1\right)}\geq0$,
which means $m\in\left(\frac{1}{2},1\right)$. On the other hand,
in order to control the term $\left(m-1\right)bY^{2}$ in the expression
(\ref{eq:E-2}) for $E_{2}$, we need $\left(m-1\right)b\geq0$. Therefore,
we have to choose $b=0$. 

It only remains to determine $\theta.$ By substituting $b=0$ into
(\ref{eq:E-2}) we have 
\[
E_{2}=\theta^{\prime}+\frac{2}{N}\theta^{2}-\frac{N}{2}\psi^{2}-\frac{2}{(m-1)\left(2m-1\right)}\left(\frac{4}{N}\theta-2\psi\right)KU.
\]
So condition (\ref{eq:20140228-2}) will be satisfied when 
\begin{equation}
\frac{4}{N}\theta-2\psi-\frac{1}{t}\geq0,\mbox{ and }\theta^{\prime}+\frac{2}{N}\theta^{2}-\frac{N}{2}\psi^{2}\geq0.\label{eq:20140306-4}
\end{equation}
By testing the coefficients of terms with the lowest degree in $t$,
we conclude that the best choice of $\left(\theta,\psi\right)$ is
$\theta=\frac{N}{2t}$ and $\psi=0$. 
\end{proof}

\section{Non-local estimates}

In this part we prove Theorem \ref{th-b}. We have seen from the last
section that when $K>0$ and $m>1$, there is no way to control the
term containing $KU$ in $E_{2}$ (\ref{eq:E-2}). To deal with this
problem, the idea is to replace $KU$ by its upper bound 
\begin{equation}
R=\sup KU\label{eq:R}
\end{equation}
and then seek for a bound on $Z$ involving constant $R$ rather than
$U$. Therefore, we choose $a=0$, and consider gradient bounds in
the form $B(t,U,Y)=\theta+bY$. By Lemma \ref{lem-1},
\begin{eqnarray*}
\left(A-\partial_{t}\right)G & \geq & \frac{2}{Nt}G^{2}+\left(\frac{4}{N}\theta-\frac{1}{t}\right)G\\
 &  & -2KUt\left(\frac{G}{t}+\theta+\left(1+b\right)Y\right)\\
 &  & +\frac{4b}{N}YG+\frac{2}{N}\left(\theta+bY\right)^{2}t+tB_{t}\\
 &  & +t(m-1)bY^{2}.
\end{eqnarray*}
 Since 
\[
\frac{G}{t}+\theta+(1+b)Y=X\geq0,
\]
we have
\begin{eqnarray*}
\left(A-\partial_{t}\right)G & \geq & \frac{2}{Nt}G^{2}+\left(\frac{4}{N}\theta-\frac{1}{t}\right)G\\
 &  & -2Rt\left(\frac{G}{t}+\theta+\left(1+b\right)Y\right)\\
 &  & +\frac{4b}{N}YG+\frac{2}{N}\left(\theta+bY\right)^{2}t+tB_{t}\\
 &  & +t\left(m-1\right)bY^{2}.
\end{eqnarray*}
Introduce a function $y$ and rearrange the terms as 
\begin{eqnarray*}
\left(A-\partial_{t}\right)G & \geq & \frac{2}{Nt}\left[G+b\left(Y-y\right)t\right]^{2}\\
 &  & +\left(\frac{4}{N}\left(\theta+by\right)-2R-\frac{1}{t}\right)G+E_{3}t
\end{eqnarray*}
 where 
\begin{eqnarray}
E_{3} & = & \left(m-1\right)bY^{2}\nonumber \\
 &  & +\left(\frac{4}{N}b\left(\theta+by\right)-2R\left(1+b\right)+b^{\prime}\right)\left(Y-y\right)\nonumber \\
 &  & +\theta^{\prime}+b^{\prime}y+\frac{2}{N}\left(\theta+by\right)^{2}-2R\left(\theta+by+y\right).\label{eq:E-3}
\end{eqnarray}
From this differential inequality, to conclude $G\leq0$ by maximum
principle, we need 
\begin{equation}
\frac{4}{N}\left(\theta+by\right)-2R-\frac{1}{t}\geq0,\mbox{ and }E_{3}\ge0.\label{eq:con}
\end{equation}

Suppose that $\left(m-1\right)b\geq0$. To ensure that $E_{3}$ is
non-negative, let us solve
\begin{equation}
\left\{ \begin{array}{c}
\theta^{\prime}+b^{\prime}y+\frac{2}{N}\left(\theta+by\right)^{2}-2R\left(\theta+by+y\right)=0\\
\frac{4}{N}b\left(\theta+by\right)-2R\left(1+b\right)+b^{\prime}=0\\
\theta\left(0\right)=\infty\\
b\left(0\right)=0.
\end{array}\right.\label{eq:ODEs}
\end{equation}
Define $C=\theta+by$. It follows from (\ref{eq:ODEs}) that 
\[
\frac{d}{dt}C+\frac{2}{N}C^{2}-2R\left(C+y\right)=0,\ C\left(0\right)=\infty.
\]
Assume that $y$ is a constant such that $y\geq-\frac{NR}{4}$. Then
we have 
\begin{equation}
C=C\left(t,y\right)=\frac{NR}{2}+\sqrt{NR}\sqrt{y+\frac{NR}{4}}\coth\frac{w\left(t,y\right)}{2}\label{eq:C}
\end{equation}
where
\[
w\left(t,y\right)=\frac{4t}{N}\sqrt{NR}\sqrt{y+\frac{NR}{4}}.
\]
By substituting it into the second equation in (\ref{eq:con}), we
can solve $b$ as 
\begin{equation}
b=\partial_{y}C\left(t,y\right)=2tR\frac{e^{2w\left(t,y\right)}-1-2e^{w\left(t,y\right)}w\left(t,y\right)}{\left(e^{w\left(t,y\right)}-1\right)^{2}w\left(t,y\right)}.\label{eq:yC}
\end{equation}
One can verify that $b\geq0$. So we have found $\left(\theta,b,y\right)$
such that $E_{3}\geq0$ when $m>1$. Moreover, $\theta+by=C$ satisfies
the first inequality in (\ref{eq:con}). Therefore, we have obtained
a family of bound on $X-Y$ with index $y\geq-\frac{NR}{4}$ as 
\[
B=\theta+bY=C\left(t,y\right)+\partial_{y}C\left(t,y\right)\left(Y-y\right).
\]
This may be organized as the following lemma.
\begin{lem}
Let $X$, $Y$, $R$, $C$ and $\partial_{y}C$ be defined according
to (\ref{04-a4}), (\ref{eq:R}), (\ref{eq:C}) and (\ref{eq:yC}).
When $m>1$, for any $y\geq-\frac{NR}{4}$, we have 
\begin{equation}
X-Y\leq C\left(t,y\right)+\partial_{y}C\left(t,y\right)\left(Y-y\right).\label{eq:20140306-5}
\end{equation}

\end{lem}
To derive Theorem \ref{th-b} from this lemma, notice that when $Y\geq-\frac{NR}{4}$,
we can let $y=Y$ in (\ref{eq:20140306-5}), hence obtaining the first
inequality in Theorem \ref{th-b}. As for the second inequality in
Theorem \ref{th-b}, it follows from (\ref{eq:C}) and (\ref{eq:yC})
that
\[
\lim_{y\rightarrow-\frac{NR}{4}}C\left(t,y\right)=\frac{NR}{2}+\frac{N}{2t}\mbox{ and }\lim_{y\rightarrow-\frac{NR}{4}}\partial_{Y}C\left(t,y\right)=\frac{2R}{3}t.
\]
Therefore, by letting $y\rightarrow-\frac{NR}{4}$ in (\ref{eq:20140306-5})
we arrived at the desiring result.


\begin{thebibliography}{10}

\bibitem{MR524760}
Donald~G. Aronson and Philippe B{\'e}nilan.
\newblock R\'egularit\'e des solutions de l'\'equation des milieux poreux dans
  {${\bf R}^{N}$}.
\newblock {\em C. R. Acad. Sci. Paris S\'er. A-B}, 288(2):A103--A105, 1979.

\bibitem{MR726106}
Philippe B{\'e}nilan, Michael~G. Crandall, and Michel Pierre.
\newblock Solutions of the porous medium equation in {${\bf R}^{N}$} under
  optimal conditions on initial values.
\newblock {\em Indiana Univ. Math. J.}, 33(1):51--87, 1984.

\bibitem{MR891781}
L.~A. Caffarelli, J.~L. V{\'a}zquez, and N.~I. Wolanski.
\newblock Lipschitz continuity of solutions and interfaces of the
  {$N$}-dimensional porous medium equation.
\newblock {\em Indiana Univ. Math. J.}, 36(2):373--401, 1987.

\bibitem{MR1814364}
David Gilbarg and Neil~S. Trudinger.
\newblock {\em Elliptic partial differential equations of second order}.
\newblock Classics in Mathematics. Springer-Verlag, Berlin, 2001.
\newblock Reprint of the 1998 edition.

\bibitem{MR797051}
Miguel~A. Herrero and Michel Pierre.
\newblock The {C}auchy problem for {$u_t=\Delta u^m$} when {$0<m<1$}.
\newblock {\em Trans. Amer. Math. Soc.}, 291(1):145--158, 1985.

\bibitem{huang2011gradient}
Guangyue Huang, Zhijie Huang, and Haizhong Li.
\newblock Gradient estimates for the porous medium equations on riemannian
  manifolds.
\newblock {\em Journal of Geometric Analysis}, 2012.

\bibitem{MR834612}
Peter Li and Shing-Tung Yau.
\newblock On the parabolic kernel of the {S}chr\"odinger operator.
\newblock {\em Acta Math.}, 156(3-4):153--201, 1986.

\bibitem{MR2487898}
Peng Lu, Lei Ni, Juan-Luis V{\'a}zquez, and C{\'e}dric Villani.
\newblock Local {A}ronson-{B}\'enilan estimates and entropy formulae for porous
  medium and fast diffusion equations on manifolds.
\newblock {\em J. Math. Pures Appl. (9)}, 91(1):1--19, 2009.

\bibitem{MR2392508}
Li~Ma, Lin Zhao, and Xianfa Song.
\newblock Gradient estimate for the degenerate parabolic equation {$u_t=\Delta
  F(u)+H(u)$} on manifolds.
\newblock {\em J. Differential Equations}, 244(5):1157--1177, 2008.

\bibitem{MR2282669}
Juan~Luis V{\'a}zquez.
\newblock {\em Smoothing and decay estimates for nonlinear diffusion
  equations}, volume~33 of {\em Oxford Lecture Series in Mathematics and its
  Applications}.
\newblock Oxford University Press, Oxford, 2006.
\newblock Equations of porous medium type.

\bibitem{MR2286292}
Juan~Luis V{\'a}zquez.
\newblock {\em The porous medium equation}.
\newblock Oxford Mathematical Monographs. The Clarendon Press Oxford University
  Press, Oxford, 2007.
\newblock Mathematical theory.

\bibitem{MR1305712}
Shing-Tung Yau.
\newblock On the {H}arnack inequalities of partial differential equations.
\newblock {\em Comm. Anal. Geom.}, 2(3):431--450, 1994.

\end{thebibliography}
\end{document}